\documentclass[12pt]{amsart}

\theoremstyle{plain}

\newtheorem{theorem}{Theorem}[section]
\newtheorem*{theoremH}{Theorem H}
\newtheorem{lemma}[theorem]{Lemma}
\newtheorem{proposition}[theorem]{Proposition}
\newtheorem{corollary}[theorem]{Corollary}
\newtheorem*{corollaryH}{Corollary H}
\newtheorem{observation}[theorem]{Observation}
\newtheorem{fact}[theorem]{Fact}

\theoremstyle{definition}

\newtheorem{definition}[theorem]{Definition}
\newtheorem{remark}[theorem]{Remark}
\newtheorem{example}[theorem]{Example}
\newtheorem{notation}[theorem]{Notation}

\newcommand{\ep}{\varepsilon}

\newcommand{\vf}{\varphi}

\newcommand{\cc}{\subset\subset}
\newcommand{\sipky}{\nearrow\!\!\!\nearrow}
\newcommand{\sipka}{\nearrow}
\newcommand{\dist}{\mathrm{dist}}

\newcommand{\R}{\mathbb{R}}
\newcommand{\N}{\mathbb{N}}

\newcommand{\Lip}{\mathrm{Lip}}

\renewcommand{\epsilon}{\varepsilon}
\renewcommand{\phi}{\varphi}
\renewcommand{\tilde}{\widetilde}

\usepackage{amssymb}
\usepackage{amsmath}
\usepackage{amsthm}

\begin{document}

\title{\large On compositions of d.c.\ 
functions\\ and mappings}

\author{Libor Vesel\'y}
\address{Dipartimento di Matematica\\
Universit\`a degli Studi\\
Via C.~Saldini 50\\
20133 Milano\\
Italy}

\author{Lud\v ek Zaj\'\i\v cek}
\address{Charles University\\
Faculty of Mathematics and Physics\\
Sokolovsk\'a 83\\
186 75 Praha 8\\
Czech Republic}

\email{vesely@mat.unimi.it}
\email{zajicek@karlin.mff.cuni.cz}



 \thanks{}

\begin{abstract}
A d.c.\ (delta-convex) function on a normed linear space 
is a function representable as a
difference of two continuous convex functions.
We show that an infinite dimensional analogue of Hartman's
theorem on stability of d.c.\ functions under compositions does not hold in general.
However, we prove that it holds in some 
interesting particular cases.
Our main results about compositions are proved in the more general context of d.c.\
mappings between normed linear spaces.
\end{abstract}

\maketitle

%%%%%%%%%%%%%%%%%%%%%%%%%%%%%%%%%%%%%%%%%%%%%%%%%%%%%%

%%%%%%%%%%%%%%%%%% Headings
\markboth{L.~Vesel\'y and L.~Zaj\'{\i}\v{c}ek}{Compositions of d.c.~mappings}
%%%%%%%%%%%%%%%%%%%%%%%%%%%

\section*{Introduction}

Let $C$ be a convex set in a (real) normed linear space $X$. A
function $f\colon C\to\R$ is called {\em d.c.}\ or {\em delta-convex}
if it can be represented as a difference of two continuous convex
functions on $C$. We say that $f$ is locally d.c.\ on $C$, if each $c\in C$ has a convex neighbourhood $U$ such that
 $f$ is d.c.\ on $U \cap C$. A mapping $F\colon C\to\R^n$ is a {\em d.c.\ mapping}
if each of its $n$ components is a d.c.\ function. There exist many articles
%(mainly in optimization theory)
which work with d.c.\ functions (see, e.g., the references in \cite{Hi} and \cite{DuVeZa}).

In 1959, P.~Hartman \cite{H} proved the following interesting well-known results.

\medskip
\noindent
(I)\ \  {\em Let $A \subset \R^m$ be a convex set which is either open or closed. 
Let $f\colon A \to \R$ be locally d.c.\ on $A$. Then $f$ is d.c.\ on $A$.}

\medskip

\noindent
(II)\ \  {\em Let $X$ be a normed linear space, $A \subset X$  a convex set which is either open or closed,
and $B \subset \R^n$ an open convex set. If $F\colon A \to B$ is a d.c.\ mapping and $g\colon B\to\R$ 
is a d.c.\ function, then the function
$g\circ F$ is locally d.c.\ on $A$.}

\medskip

In fact, Hartman \cite{H} formulated (II) only for the case $X = \R^m$, but he mentioned 
(see the end of p.707) that his proof
 clearly works also in  more general settings (we could even suppose that $X$ is a topological linear space and $A$ 
is an arbitrary convex set). For a generalization of (II), proved in a
quite different way, see Proposition~\ref{lokdc}.

  Hartman also remarked that his proof of (I) does not work for
infinite dimensional spaces. A corresponding counterexample was provided
by E.~Ko\-peck\'a and J.~Mal\'y \cite{KM}: given a nonempty open convex
set $A\subset\ell_2$, there exists a locally d.c.\ function on $A$ which
is not d.c.\ on $A$. (They also remark without proof that a similar
example can be constructed in each infinite dimensional normed linear space;
we prove this claim in Corollary~\ref{jldc}.)

The results (I) and (II) immediately imply the following superposition theorem.

\begin{theoremH}\label{T:H}
Let $A\subset \R^m$ and $B\subset\R^n$ be convex sets. Let  $A$ be either open or closed, and let $B$ be open. If
$F\colon A\to B$ and $g\colon B\to\R$ are d.c., then the function
$g\circ F$ is d.c.
\end{theoremH}

\noindent

Note that Hartman did not mention Theorem~H explicitly, but he formulated its corollary
 (obtained by putting $F := (f_1,f_2)$ and $g(x,y):= xy$ or $g(x,y):= x/y$):

\begin{corollaryH}\label{sopo}
Let $A \subset \R^m$ be either an open or a closed convex set. Let $f_1$, $f_2$ be d.c.\ on $A$. Then the
 product $f_1\!\cdot\! f_2$ and, if $f_2(x) \neq 0$ for $x \in A$, the quotient $f_1/f_2$ are d.c.\ functions
  on $A$.
  \end{corollaryH}

 Note that the case of the product can be proved in a more elementary way 
(see \cite{Hi}), but the 
   stability with respect to quotients probably cannot be proved more easily.

Though (I) cannot be used  to generalize  Theorem H
to infinite dimensions, it remained open whether such a generalization is possible.
The present paper concerns this question.
We show that an infinite dimensional analogue of Theorem H does not hold 
(see Corollary~\ref{jldc}): 

 \medskip
 
\noindent
{\em For each infinite dimensional normed linear space $X$, there exists a positive d.c.\ function
$f$ on $X$ such that $1/f$ is not d.c.}

\medskip
 
 However, using a modification of Hartman's methods, we prove (Theorem
\ref{cely}) the following variant of Theorem~H (for other variants see Theorem~\ref{spec}).

 \medskip
 
\noindent {\it
Let $X$ be a normed linear space.
Let $A\subset X$ be an open convex set, and
$F\colon A\to \R^n$ and $g\colon \R^n\to\R$ be d.c. Then the function
$g\circ F$ is d.c.}
\smallskip

\noindent
 Consequently, if $f$, $h$ are d.c.\ on $A$, then, for instance, 
$\exp (f)$ and $\frac{fh}{1+f^2 + h^2}$
  are d.c.\ on $A$ (see the text after Theorem~\ref{cely}).

Another positive result in which $F$ is a real continuous convex (or concave)
function is Proposition~\ref{reflex}. It implies (see Remark~\ref{remreflex}(i)) 
the following: 

\medskip\noindent
{\em Let $X$ be a reflexive Banach space and $f_1$, $f_2$ be continuous convex functions on 
 $X$. If the quotient $f_1/f_2$ is defined on $X$, then it is d.c.}

\medskip\noindent
(Note that the above statement is true only in reflexive spaces, see \cite{HKVZ}.)

We prove our results in a more general context of d.c.\ mappings between
normed linear spaces. In particular, we prove (see Corollary~\ref{bilin})
that, in some interesting cases, the inner product (and even a general ``product''
given by a bilinear mapping) of two d.c.\
mappings is d.c.\ as well.

%%%%%%%%%%%%%%%%%%%%%%%%%%%%%%%%%%%%%%%%%%%%%%%%%%%%%%%%%%%%%%%%%%%%%%%%%%%%%%%%%

\section{Preliminaries}\label{S:prelim}

We consider only normed linear spaces over the
reals $\R$.  If $X$ is a normed linear space, we denote by $B_X$ its closed unit ball.
By $B(x,r)$ we denote the open ball with center
$x$ and radius $r$. We say that a Lipschitz mapping
 $F$ is $L$-Lipschitz, if $\Lip F \leq L$, where $\Lip F$ is the (least) Lipschitz constant of $F$.

\begin{definition}[\cite{VeZa}]\label{D:dc}
Let $X,Y$ be normed linear spaces, $C\subset X$ be a convex set, and
$F\colon C\to Y$ be a continuous mapping. We say that $F$ is {\em d.c.}\
(or {\em delta-convex}) if there exists a continuous (necessarily convex)
function $f\colon C\to\R$ such that $y^*\circ F+f$ is convex on $C$
whenever $y^*\in Y^*$, $\|y^*\|\le1$. In this case we say that $f$
controls $F$, or that $f$ is a {\em control function} for $F$.
\end{definition}

\begin{remark}\label{R:dc}
The following facts are easy to prove (cf.~\cite{VeZa}).
\begin{enumerate}
\item[(a)]  For $Y=\R^n$, the above definition of a d.c.\ mapping
coincides with the one in
the beginning of Introduction. Moreover, if $F=(F_1,\dots,F_n)$ and $f_i$ controls $F_i$, then
 $f:= f_1+\dots+f_n$ controls $F$.
\item[(b)] If $g=f_1-f_2$, where $f_1,f_2$ are continuous convex
functions on a convex subset of a normed linear space, then $f_1+f_2$
controls $g$. 
\item[(c)] The notion of delta-convexity does not depend on the choice
of equivalent norms on $X$ and $Y$. 
\end{enumerate}
\end{remark}

A theory of d.c.\ mappings on open convex sets was developed in \cite{VeZa}. Some further
results, together with a survey of main results from \cite{VeZa}, can be
found in \cite{DuVeZa}.
We shall need the following two propositions.

\begin{proposition}[\cite{VeZa}]\label{P:veza}
Let $X,Y,Z$ be normed linear spaces, and let $A\subset X$ and $B\subset Y$ be
convex sets. Let $F\colon A\to B$ and $G\colon B\to Z$ be d.c.\ mappings
with control functions $f\colon A\to\R$ and $g\colon B\to \R$, respectively. If $G$ and
$g$ are Lipschitz on $B$,
then $G\circ F$ is d.c.\ on $A$ with a
control function $h=g\circ F +(\Lip G+ \Lip g)f$.
\end{proposition}

\begin{proof}
This was proved in \cite[Proposition~4.1]{VeZa}  assuming that the sets
$A,B$ are also open, since this was the context the authors were interested in.
However, it is easy to see that the proof does not need this additional
assumption. Indeed, the proof is based on the equivalence of (i) and (iii) in 
\cite[Proposition 1.13]{VeZa}, whose
proof does not use  the openness of $A$. 
\end{proof}

\begin{proposition}\label{P:lip}
Let $X,Y$ be normed linear spaces, $C\subset X$ a bounded open convex
set, and $F\colon C\to Y$ a d.c.\ mapping with a Lipschitz control
function. Then $F$ is Lipschitz.
\end{proposition}

\begin{proof}
This was stated in \cite[Theorem~18(i)]{DuVeZa}  for $X$ and $Y$ Banach
spaces, but the proof therein works for normed linear spaces as well. 
(Note that the question for which open convex
 sets $C$ the proposition holds was answered in \cite{CsNa}.)
\end{proof}

\begin{notation}
Let $A,B,A_n,B_n$ ($n\in\N$) be subsets of a normed linear space $X$. We
shall use the notation:
\begin{itemize}
\item $A\subset\subset B$ whenever there exists $\epsilon>0$
such that $A+ B(0,\varepsilon) \subset B$;
\item $A_n\nearrow A$ whenever $A_n\subset A_{n+1}$ for each
$n\in\N$, and $\bigcup_{n\in\N}A_n=A$;
\item $A_n\nearrow\!\!\!\nearrow A$ whenever $A_n\subset\subset A_{n+1}$ for each
$n\in\N$, and $\bigcup_{n\in\N}A_n=A$.
\end{itemize}
\end{notation}

\begin{fact}\label{F:1}
Let $C\subset X$ be a nonempty convex set, and $f\colon C\to\R$  a
convex function.
\begin{enumerate}
\item[(a)] If $C$ is open and bounded, and $f$ is continuous,
           then $f$ is bounded below on $C$.
\item[(b)] If $f$ is bounded on $C$, then $f$ is Lipschitz on each $D\subset\subset C$.
\item[(c)] If $f$ is $L$-Lipschitz on $C$, then $f$ admits a convex $L$-Lipschitz
extension to the whole $X$.
\end{enumerate}
\end{fact}

\begin{proof}
(a) follows from the fact that $f$ is minorized by a continuous affine
function (by the Hahn-Banach theorem).
\\
(b) can be proved in the same way as local
Lipschitz continuity of
continuous, convex functions. For the sake of completeness, we give a
sketch of proof. Let $|f|\le M$ on $C$, $r>0$ be such that
$D+ B(0,2r) \subset C$, and $x,y\in D$. Then $z:=y+\frac{r(y-x)}{\|y-x\|}\in
C$, and $y=\frac{r}{\|y-x\|+r}x+\frac{\|y-x\|}{\|y-x\|+r}z$. By convexity,
$f(y)\le\frac{r}{\|y-x\|+r}f(x)+\frac{\|y-x\|}{\|y-x\|+r}f(z)$. It
easily follows that
$f(y)-f(x)\le\frac{\|y-x\|\bigl(f(z)-f(y)\bigr)}{r}\le\frac{2M}{r}\|y-x\|$.
The rest follows by interchanging $x$ and $y$.
\\
(c) It is well-known (and easy-to-prove)
that the function $\widehat{f}\colon X\to\R$, given by
$\widehat{f}(x)=\inf\bigl\{f(c)+L\|x-c\|:c\in C\bigr\}$,
is a convex, $L$-Lipschitz extension of $f$ (cf.\ \cite{CoMu}).
\end{proof}

We shall need the following well-known and very easy fact.

\begin{fact}\label{obalky}
Let $C$ be a convex set in a normed linear space $X$, and $r>0$.
Then the sets (called ``inner parallel set'' and ``outer parallel set'' of $C$)
$$
D:=\{x\in C:\mathrm{dist}(x,X\setminus C)>r\},\ \ 
E:=\{x\in X: \mathrm{dist}(x,C)<r\}
$$
are  convex.
\end{fact}

\begin{observation}\label{O:lip}
Let $X,Y$ be normed linear spaces, $C\subset X$ a convex set,
and $F\colon C\to Y$ a d.c.\ mapping with a bounded above
control function $f$. Then both $F$ and $f$ are Lipschitz on each
bounded convex set $B\cc C$.
\end{observation}

\begin{proof}
By Fact~\ref{obalky},
there exist open, bounded, convex sets
$D$ and $E$ such that $B\subset D\cc E\subset C$. By Fact~\ref{F:1}(a),
$f$ is bounded on $E$. Hence $f$ is Lipschitz on $D$ by
Fact~\ref{F:1}(b), and $F$ is Lipschitz on $D$ by
Proposition~\ref{P:lip}.
\end{proof}

\begin{definition}
A normed linear space $X$ is said to have {\em modulus of convexity of power
type~2} if there exists $a>0$ such that $\delta_X(\epsilon)\ge
a\epsilon^2$ for each $\epsilon\in(0,2]$ (where $\delta_X$ denotes the
classical modulus of convexity of $X$; see, e.g., \cite{day} for the
definition).
\end{definition}

\begin{fact}\label{F:powertype}\
\begin{enumerate}
\item[(a)] The $\ell_2$-direct sum $(X\oplus Y)_{\ell_2}$ has modulus of convexity of power
type 2 whenever both $X$ and $Y$ do.
\item[(b)] All $L_p(\mu)$ spaces with
$1<p\le2$  ($\mu$ arbitrary nonnegative measure) have modulus of
convexity of power type 2 (in their canonical norms).
\end{enumerate}
\end{fact}

\begin{proof}
(a) follows immediately from the following result by Bynum
\cite{bynum}: $X$ has modulus of convexity of power
type 2 if and only if there exists $b>0$ such that
$2\|x\|^2+2\|y\|^2\ge\|x+y\|^2+b\|x-y\|^2$ for each $x,y\in X$.\\
(b) is due to Hanner \cite{hanner}.
\end{proof}

Let $X,Y$ be normed linear spaces, and $A\subset X$
 an open set. Recall that a mapping $F\colon A\to Y$ is said to be
$C^{1,1}$ on $A$ if its
Fr\'echet
derivative $F'(x)$ exists at each point $x\in A$ and 
$F'\colon A\to\mathcal{L}(X,Y)$ is
Lipschitz. 

The next proposition follows from the proof of the implication $(i)\!\Rightarrow\!(ii)$
in \cite[Theorem~11]{DuVeZa}.

\begin{proposition}\label{C11}
Let $X,Y$ be normed linear spaces,
$A\subset X$ an open convex set, $F\colon A\to Y$  a $C^{1,1}$
mapping. If $X$ admits an equivalent norm $|\cdot|$ with modulus of convexity
of power type 2,  then $F$ is
d.c.\ on $A$ with a control function of the form
$f(x)=c|\cdot|^2$  for some $c>0$.
\end{proposition}

%%%%%%%%%%%%%%%%%%%%%%%%%%%%%%%%%%%%%%%%%%%%%%%%%%%%%%%%%%%%%%%%%%%%%%%
\section{A consequence of Hartman's construction}

Hartman's construction \cite{H}, which gives the proof that
locally d.c.\ functions
in $\R^n$ are d.c., has some consequences also in infinite dimensional
spaces. It was observed (independently) already in \cite{PB} and
\cite{KM} (cf.\ Remark~\ref{obec}). The main new observation of the
present article is that Hartman's construction gives even a
characterization of d.c.\ mappings on open sets (Proposition~\ref{P:HKM}) which
(together with Proposition~\ref{P:veza}) implies some
infinite dimensional versions of Hartman's superposition theorem.
First we formulate a lemma which describes Hartman's construction in a general setting.

\begin{lemma}\label{L:hartman}
Let $X,Y$ be normed linear spaces, $C\subset X$ a nonempty
convex set, and  $F\colon C\to Y$ a mapping. Let $\emptyset \neq D_n \subset C$
($n\in\N$)
be convex sets such that $D_n\sipka C$ and, for each $n$, 
$\mathrm{dist}(D_n,C\setminus D_{n+1})>0$, 
$D_n$ is relatively open in $C$,  
and
$F|_{D_n}$ is d.c.\ with a
control function  $\gamma_n\colon D_n\to \R$
which is either bounded or Lipschitz. Then $F$ is d.c.\ on $C$.
\end{lemma}

\begin{proof}
First, fix  $a\in D_1$, and observe
that the bounded sets $\tilde{D}_n:=D_n\cap B(a,n)$ satisfy the same
assumptions as the sets $D_n$. Thus we can (and do) suppose that each
$D_n$ is bounded, and hence each $\gamma _n$ is bounded on $D_n$.
 Adding a constant to $\gamma_n$ if necessary, we can suppose that $0< \gamma_n(x) < b_n < \infty$ for
 each $n \in \N$ and $x \in D_n$. 

For each $n \in \N$, choose $0 < d_n < \mathrm{dist}(D_n,C\setminus D_{n+1})$, and
 consider the Lipschitz convex functions  
$\vf_n(x) := \frac{b_n+1}{d_n} \, \dist(x,D_n)$ on $C$. Define
 $$ h_n(x):= \max\{\gamma_{n+1}(x), \vf_n(x)\},\   x \in  D_{n+1},\ \ \text{and}\ \ 
 h_n(x):= \vf_n(x),\   x\in C \setminus D_{n+1}.$$
 If $z \in D_{n+1}$, then there exists $\ep>0$ such that $h_n(x)= \max\{\gamma_{n+1}(x), \vf_n(x)\}$ for
  $x \in C \cap B(z,\ep)$, since $D_{n+1}$ is open in $C$. If $z \in C \setminus D_{n+1}$, then 
$\mathrm{dist}(x,D_n)\ge d_n$ and therefore
there exists
$\ep>0$ such that   $\vf_n(x) > b_n$, and thus  $h_n(x)=  \vf_n(x)$, for each
  $x \in C \cap B(z,\ep)$.
  Therefore, $h_n$ is continuous and convex on $C$. Moreover, clearly
\begin{itemize}
\item  $h_n \geq 0$, and $h_n$ is bounded on each bounded subset of $C$.
\end{itemize}
Since $\vf_n(x) =  0 $  for  $x \in D_n$, we see that 
$h_n(x) = \gamma_{n+1}(x)$ for $x \in D_n$. So,
\begin{itemize}
\item $h_n$ is a control function of $F$ on $D_n$.
\end{itemize}
%We also can (and do) suppose that
%\begin{itemize}
%\item $h_n\ge0$ (this can be achieved by adding an
%appropriate continuous affine function).
%\end{itemize}

Let us define, by induction, a sequence $\{f_n\}$
of continuous convex  functions on
$C$ such that:
\begin{enumerate}
\item[(a)] $f_n$ is bounded on bounded subsets of $C$,
\item[(b)] $f_n\ge0$,
\item[(c)] $f_n$ controls $F$ on $D_{n+1}$, and 
\item[(d)] $f_{n+1}=f_n$ on $D_n$.
\end{enumerate}
Put $f_1:=h_2$. Suppose we already have $f_1,\ldots,f_n$. Set
\[
s:=\sup f_n(D_{n+2})\,,\qquad \sigma:=\sup h_{n+2}(D_n)\, \ \ \text{and}
\]
%Let $r>0$ be such that $\dist(D_n, C\setminus D_{n+1})> r$. Put
\[
g_n(x):=h_{n+2}(x)-\sigma + \frac{\sigma+ s+1}{d_n}\dist(x,D_n)
\ \ \ \text{for}\ \ x\in C.
\]
Then clearly $g_n$ is continuous and convex on C, and it controls $F$ on $D_{n+2}$.
Define $f_{n+1}=\max\{f_n,g_n\}$. Clearly $f_{n+1}$ is continuous convex,  $f_{n+1}\ge f_{n}\ge0$,
and $f_{n+1}$ is bounded on bounded subsets of $C$. If $x \in D_n$, then $g_n(x)\le 0\le f_n(x)$, consequently
 $f_{n+1}=f_n$ on $D_n$.

Let us show that $f_{n+1}$ controls $F$ on
$D_{n+2}$; i.e., that the function
$\varphi_{y^*}:=y^*\circ F+f_{n+1}=\max\{y^*\circ F+f_n,y^*\circ F+g_n\}$ is
continuous and convex on $D_{n+2}$ for each  $y^*\in B_{X^*}$.
 To this end, fix $y^*\in B_{X^*}$ and  $z \in D_{n+2}$. 
If $z \in D_{n+1}$, then there is $\ep>0$
  such that $\varphi_{y^*}$ is continuous and convex on $B(z,\ep) \cap C$ (since $D_{n+1}$ is open in $C$ and both
   $f_n$ and $g_n$ control $F$ on $D_{n+1}$). If $z \in D_{n+2} \setminus D_{n+1}$, then $\dist(z, D_n) \geq d_n$, and
    consequently $g_n(x)\geq 0-\sigma+(\sigma+s+1) > f_n(x)$. Therefore there exists $\ep>0$
  such that $U:=B(z,\ep) \cap C\subset D_{n+2}$ and     $\varphi_{y^*}$ equals to the continuous convex function
  $y ^*\circ F+g_n$  on  $U$. Hence we can conclude that  $\varphi_{y^*}$  is
continuous and convex on $D_{n+2}$.

Now, for each $x\in C$, the sequence $\bigl\{f_n(x)\bigr\}$ is constant
for large $n$'s, hence
$f(x):=\lim_{n\to\infty} f_n(x)$ is well defined on $C$. 
Since $f=f_n$ on $D_n$, (c) easily implies that
$f$ is a continuous convex function, which controls $F$ on $C$.
\end{proof}

\begin{remark}\label{zahl}
%\begin{enumerate}
%\item 
The assumptions of Lemma \ref{L:hartman} allow the possibility 
that $D_n = D_{n+1}= \dots=C$ for some $n$.
%\item 
%If each $\gamma_n$ is Lipschitz, then the above construction gives 
%a control function $f$ which is Lipschitz on each $D_n$.
% \end{enumerate}
 \end{remark}

\begin{lemma}\label{L:dn}
Let X be a normed linear space and
let $C\subset X$ be nonempty, open and convex. Let $\{C_n\}$ be a sequence of
convex sets with nonempty interior, such that $C_n\sipka C$.
Then there exists a sequence $\{D_n\}$ of nonempty
bounded open
convex sets such that $D_n\sipky C$, and $D_n\subset\subset C_n$ for each $n$.
\end{lemma}

\begin{proof}
We can (and do) suppose that each $C_n$ is bounded. (If this is not the case,
replace, for each $n$, the set $C_n$ with the set $C_n\cap B(x_0,n)$
where $x_0$ is an arbitrary interior point of $C_1$.)
\\
First we claim that $C=\bigcup_n\mathrm{int}\, C_n$. Indeed, let $x\in C$
be any point. Then $x\in C_n$ for some $n$. If $x\notin\mathrm{int}\, C_n$,
choose any $y\in\mathrm{int}\, C_n$. There exists $z\in C$ such that
$x\in(y,z)$ (i.e., $x$ is a relative interior point of the segment
$[y,z]$). There exists $k>n$ such that $z\in C_k$. Then
$x\in\mathrm{int}\, C_k$, since $y\in\mathrm{int}\, C_k$.
\\
Now, fix $\delta>0$ such that $C_1$ contains an open ball of radius $2\delta$, and
define
$$D_n:=\{x\in C_n : \dist(x, X\setminus C_n)>\delta/n\}.$$
Obviously $D_n\subset\subset C_n$ for each $n$, and the sets $D_n$ are
nonempty, open and (by Fact~\ref{obalky}) convex. Moreover
\[
D_n\subset\subset
\{x\in C_n: \dist(x, X\setminus C_n)>\delta/(n+1)\}\subset
D_{n+1}.
\]
To finish the proof, fix $x\in C$. Then $x\in\mathrm{int}\, C_n$ for some $n$.
Fix $k>n$ such that
$\dist(x,X\setminus C_n)>\delta/k$. Then
$\dist(x,X\setminus C_k)\ge\dist(x,X\setminus C_n)>\delta/k$
which means that $x$ belongs to $D_k$.
\end{proof}

Now, we are ready to state the main result of this section.

\begin{proposition}\label{P:HKM}
Let $X,Y$ be normed linear spaces, $C\subset X$ a nonempty open
convex set, and $F\colon C\to Y$ a mapping. Then the following
assertions are equivalent:
\begin{enumerate}
\item[(i)]
$F$ is d.c.\ on $C$;
\item[(ii)]
there exists a sequence $\{C_n\}$ of convex sets with nonempty interior
such that $C_n\sipka C$ and, for each $n$, $F|_{C_n}$ is d.c.\ with a control
function that is bounded from above on $C_n$;
\item[(iii)]
there exists a sequence $\{D_n\}$ of bounded open convex sets such that
$D_n\sipky C$ and, for each $n$, $F|_{D_n}$ is Lipschitz and
d.c.\ with a Lipschitz control function
on $D_n$.
\end{enumerate}
\end{proposition}

\begin{proof}
$(i)\Rightarrow(ii).$
Let $f\colon C\to\R$ be a control function for
$F$. Fix $x_0\in C$ and consider the sets
$C_n=\{x\in C: f(x)< f(x_0)+n\}$ ($n\in\N$). They are nonempty, open and
convex, and they obviously satisfy (ii).

\smallskip\noindent
$(ii)\Rightarrow(iii).$
Let $\{C_n\}$ be as in (ii). Let $D_n$ ($n\in\N$) be the bounded, open,
convex sets constructed in Lemma~\ref{L:dn} from the sets $C_n$.
Then (iii) follows immediately from Observation~\ref{O:lip}.

\smallskip\noindent
$(iii)\Rightarrow(i)$ follows from Lemma~\ref{L:hartman}.
\end{proof}

Proposition~\ref{P:HKM} easily implies the following 
generalization of Hartman's result (I) from Introduction,
which was stated (for open $A$) already in \cite[Theorem~1.20]{VeZa}
with only a hint for the proof.

\begin{corollary}\label{C:loc}
Let $A\subset\R^d$ be a convex set which is either open or closed, and let $Y$ be a normed linear
space. Then each locally d.c.\ mapping $F\colon A\to Y$ is d.c.\ on $A$.
\end{corollary}

\begin{proof}
First we will show that $F$ is d.c.\ on each compact convex
set $C \subset A$. Using compactness of $C$ and \cite[Lemma 1]{H},
 we easily see that there exist continuous convex functions $f_i$ on $A$, $x_i \in C$, and $r_i>0$, $i=1,\dots,k$, such that 
  $C \subset \bigcup_{i=1}^k B(x_i,r_i)$ and $f_i$ controls $F$ on $C \cap B(x_i,r_i)$. Consequently, $f_C 
= f_1+\dots+f_k$ controls $F$ on $C$.
   
   Now, distinguish two cases. First suppose that $A$ is open. 
Then choose compact convex sets $C_n$ with nonempty interior such that
$C_n\sipka A$. Since $f_{C_n}$ is bounded on $C_n$, Proposition \ref{P:HKM} implies that $F$ is d.c.

If $A$ is closed, choose $z \in A$ and put $D_n:= A \cap B(z,n)$. 
Since $\overline{D_n}\subset A$ is  compact and convex, $F$ is d.c.\ on $D_n$ (with a bounded control function), and we can apply 
Lemma \ref{L:hartman}. 
\end{proof}

\begin{remark}\label{obec}
It is known (see \cite{BFV}) that, on each infinite dimensional
Banach space, there exists
a continuous convex function which is unbounded on a ball. This implies
(via Fact~\ref{F:1}(b) and Proposition~\ref{P:lip}) that the implication
$(ii)\Rightarrow(i)$ in Proposition~\ref{P:HKM} is a  {\it strict} generalization of
both  \cite[Theorem~2.3]{PB} and \cite[Corollary~18]{KM}, where
delta-convexity of $F$ was proved under the following stronger
assumption: $F$ is d.c.\ on each bounded closed convex $B\subset C$
with a Lipschitz (\cite{PB}) or bounded (\cite{KM}) control function on
$B$.
\end{remark}

%%%%%%%%%%%%%%%%%%%%%%%%%%%%%%%%%%%%%%%%%%%%%%%%%%%%%%%%%%%%%%%%%%%%%%%%

\section{Global delta-convexity of composed mappings}

Let us start with the following generalization of (II) (see
Introduction) which is essentially proved in \cite[Theorem~4.2]{VeZa}.

\begin{proposition}\label{lokdc}
Let $X,Y,Z$ be normed linear spaces, $A\subset X$ a convex set, and
$B\subset Y$ an open set. Let $F\colon A\to B$ and $G\colon B\to Z$ be
locally d.c.\ mappings. Then $G\circ F$ is locally d.c.
\end{proposition}

\begin{proof}
Fix $a\in A$. Since $G$ is locally d.c.\ and each d.c.\ mapping on an
open convex subset of $Y$ is locally Lipschitz (see
\cite[Proposition~1.10]{VeZa}), there exists an open convex neighborhood
$B_0\subset B$ of $F(a)$ on which $G$ is Lipschitz and d.c.\ with a
Lipschitz control function. Find $\delta>0$ such that, for
$A_0:=B(a,\delta)\cap A$, we have that $F(A_0)\subset B_0$ and $F|_{A_0}$
is d.c. Then $G\circ F|_{A_0}=(G|_{B_0})\circ (F|_{A_0})$ is d.c.\ by
Proposition~\ref{P:veza}.
\end{proof}

Our results on global delta-convexity of composed mappings will follow from the
next basic lemma.

\begin{lemma}\label{L:main}
Let $X,Y,Z$ be normed linear spaces, let $A\subset X$ 
and $B\subset Y$ be convex sets, and let $F\colon A\to B$ 
 and $G\colon B\to Z$ be mappings. 
Suppose
there exist sequences  of convex sets $A_n \subset A$, $B_n \subset B$
such that $F(A_n) \subset B_n$, $G|_{B_n}$ is Lipschitz and d.c.\ with a 
Lipschitz control function,
 and at least one of the following conditions holds:
 \begin{enumerate}
 \item
 $A_n$ is relatively open in $A$, $A_n\sipka A$, $\dist(A_n, A \setminus A_{n+1}) >0$, 
  $F|_{A_n}$ is either bounded or Lipschitz and it is d.c.\ with a control function
which is either bounded or Lipschitz.
  \item
  $A$ is open, $F$ is d.c., $\mathrm{int} A_n \neq \emptyset$, and $A_n\sipka A$.
  \end{enumerate}
 Then $G\circ F$ is d.c.\ on $A$.
 \end{lemma}

\begin{proof}
Let (i) hold. As in the proof of Lemma~\ref{L:hartman}, we can (and do)
suppose that the sets $A_n$ are bounded. Then, on each $A_n$, $F$ is
bounded and admits a bounded control function. 
Proposition~\ref{P:veza} implies that the mapping
$G \circ F|_{A_n}= (G|_{B_n})
\circ (F|_{A_n})$ is d.c.\ with a bounded control function. By
Lemma~\ref{L:hartman}, $G\circ F$ is d.c.
  
  Now, suppose that (ii) holds. By Lemma~\ref{L:dn}, we can (and do)
suppose that $A_n\sipky A$ and each $A_n$ is open.
By Proposition~\ref{P:HKM},
there exists a sequence $\{D_n\}$ of bounded, open, convex sets such that
$D_n\sipky A$ and, for each $n$, $F|_{D_n}$ is Lipschitz and d.c.\
 with a Lipschitz control function. Then the sets $\tilde{A}_n:=A_n\cap D_n$ are open and convex, 
$F(\tilde{A}_n) \subset B_n$, and $\tilde{A}_n\sipky A$. Thus the condition (i) holds with
$A_n$ replaced by $\tilde{A}_n$. So $G \circ F$ is d.c by the first part of the proof.
\end{proof}

As a simpler but still rather general consequence we obtain:

\begin{proposition}\label{nove}
Let $X,Y,Z$ be normed linear spaces, let $A\subset X$ 
and $B\subset Y$ be convex sets, and let $F\colon A\to B$ 
 and $G\colon B\to Z$ be mappings. Suppose that 
the restriction of $G$ to each bounded convex subset of $B$
is Lipschitz and 
d.c.\ with a Lipschitz control function,  
and at least one of the following conditions holds.
  \begin{enumerate}
  \item
The restriction of $F$ to each bounded convex subset of $A$
is bounded and d.c.\ with  a bounded control function.
  \item
  $A$ is open and $F$ is d.c.
  \end{enumerate}
Then $G \circ F$ is d.c.
\end{proposition}

\begin{proof}
To prove (i), choose an arbitrary $a\in A$ and,
for each $n\in\N$, set $A_n:=B(a,n)\cap A$,
$B_n:=\mathrm{conv}\,F(A_n)$. It is easy to see that
$\mathrm{dist}(A_n,A\setminus A_{n+1})>0$ and 
$B_n\subset B$ is bounded for each $n$. Thus
$G \circ F$ is d.c.\ by Lemma \ref{L:main}.

To prove (ii), use Proposition~\ref{P:HKM} to choose a sequence $\{A_n\}$ of bounded open convex sets such that
 $A_n\sipky A$ and, for each $n$, $F|_{A_n}$ is Lipschitz and d.c.\
 with a Lipschitz control function. Then $B_n := \mathrm{conv} F(A_n)$ is clearly bounded and convex, and 
  thus $G|_{B_n}$ is Lipschitz and d.c.\
 with a Lipschitz control function. Apply Lemma \ref{L:main}.
\end{proof}

Most of the next results are corollaries of Proposition~\ref{nove}. 
One of the 
exceptions is the following 
interesting proposition.

\begin{proposition}\label{reflex}
Let $C$ be an open convex subset of a reflexive Banach space $X$,
and $f\colon C\to \mathbb{R}$ be a continuous convex function. Let $I\subset\mathbb{R}$ be an open interval containing $f(C)$. 
Then, for every normed
linear space $Z$ and every d.c.\ mapping $G\colon I\to Z$, the
composed map $G\circ f$ is d.c.\ on $C$.
\end{proposition}

\begin{proof}
Let $\{b_n\}\subset\bigl(\inf f(C),\sup I\bigr)$ be an increasing
sequence tending to $\sup I$. Then clearly the sets
$C_n:=\{ x\in C: f(x)< b_n\}$ are nonempty, open and convex, and
$C_n\sipka C$. By Lemma~\ref{L:dn}, there exist nonempty
bounded open  convex sets $D_n\cc C_n$ ($n\in\mathbb{N}$) with
$D_n\sipky C$. 
Since $f$ attains its infimum on 
the weakly compact set $\overline{D_n}$
(see e.g.\ \cite[Theorem~25.1(b)]{Deimling}), 
we have
$a_n:=\min f\bigl(\overline{D_n}\bigr)>\inf{I}$  and hence
$f(D_n)\subset[a_n, b_n]\subset I$ ($n\in\N$). Since $G$ and its control function
are locally Lipschitz on $I$ (cf.\ \cite[Proposition~1.10]{VeZa}), they are Lipschitz on each
$[a_n, b_n]$. Apply Lemma~\ref{L:main} with $A:= C$, $A_n:=D_n$, and $B_n:= [a_n,b_n]$.
\end{proof}

\begin{remark}\label{remreflex}
\begin{enumerate}
\item
Proposition~\ref{reflex} implies that $1/f$ is d.c.\ whenever $f$ is a positive continuous 
convex function on an open convex subset of a reflexive Banach space.
\item
It is easy to see that
Proposition~\ref{reflex} holds for concave (instead of convex) $f$ as
well. However it is not true for {\em all} d.c.\ functions $f$ (see Corollary~\ref{jldc}).
\item
Proposition~\ref{reflex} fails in any nonreflexive Banach space $X$: by \cite{HKVZ},
a Banach space
$X$ is reflexive
 if and only if $1/f$ is d.c.\ for each positive continuous convex function $f$ on $X$. 
\end{enumerate}
\end{remark}

\begin{theorem}\label{T:best}
Let $X,Y,Z$ be normed linear spaces, let $A\subset X$ and $B\subset Y$ be
open convex sets, and let $F\colon A\to B$ and $G\colon B\to Z$ be d.c.\
mappings. Then $G\circ F$ is d.c.\ on $A$, provided
at least one of the following
conditions is satisfied:
\begin{enumerate}
\item[(a)] $B=Y$ and $G$ admits a control function $g$ that is bounded on
bounded sets;
\item[(b)] $Y$ is finite-dimensional and $\overline{F(A)}\subset B$;
\item[(c)] $Y$ admits a renorming with modulus of convexity of power
type 2, and $G$ is $C^{1,1}$ on bounded open subsets of $B$.
\end{enumerate}
\end{theorem}

\begin{proof}
Let (a) hold. Let $E \subset Y$ be an arbitrary bounded convex set. Choose a bounded convex set $C$ such that
 $E \cc C$. Since $g$ is bounded on $C$, Observation~\ref{O:lip} implies
that both $G$ and $g$ are Lipschitz on $E$. Thus $G \circ F$
 is d.c.\ by Proposition \ref{nove} .

Now, suppose (b) holds. By Proposition~\ref{P:HKM}, there exists a sequence
$\{A_n\}$ of nonempty bounded open convex sets such that
$A_n\sipka A$ and $F$ is Lipschitz on each $A_n$.
Since each
 $\overline{F(A_n)}$
is a compact subset of $B$ ($Y$ is finite-dimensional!),  
 $B_n:=\mathrm{conv}\,\overline{F(A_n)}\cc B$ is a compact convex subset of $B$.
 Let $\tilde g$ be a control function of $G$. We can clearly find $\ep>0$ such that $\tilde g$ is bounded
  on $C:= B_n + B(0,\ep) \subset B$.  
Observation~\ref{O:lip} implies
that both $G$ and $\tilde g$ are Lipschitz on $B_n$.
Now, Lemma~\ref{L:main} shows that $G\circ F$ is d.c.

Finally, let (c) hold.  For each bounded convex set $E\subset B$, let
$B_0\subset B$ be a bounded convex open set
containing $E$. Since $G$ is $C^{1,1}$ on $B_0$, it
is also Lipschitz on $B_0$. Moreover, 
%by \cite{DuVeZa}
%(proof of Theorem~11, $(i)\Rightarrow(ii)$), $G$ is controlled on $B_0$
%by a function of the form
%$g=c|\cdot|^2$ where $c>0$ is a constant and $|\cdot|$
%is an equivalent norm on $Y$. 
Proposition~\ref{C11} easily implies that $G$ admits a Lipschitz control function on $B_0$, and
hence also on $E$. Thus, we can apply Proposition~\ref{nove}.
\end{proof}

Let $X,Y$ be vector spaces.
Recall that a mapping $Q\colon X\to Y$ is {\em quadratic} if there
exists a bilinear mapping $B\colon X\times X\to Y$ such that
$Q(x)=B(x,x)$ for each $x\in X$. In this case, we say that $Q$ {\em is generated} by $B$. 

\begin{definition}[\cite{KoVe}]
A normed linear space $X$ is said to have the {\em property~(D)} if every
continuous quadratic form on $X$ can be represented as a difference of
two nonnegative continuous quadratic forms.
\end{definition}

\begin{proposition}\label{quadr}
Let $X,Y,Z$ be normed linear spaces, $C\subset X$ an open convex set,
$F\colon C\to Y$ a d.c.\ mapping, and $Q\colon Y\to Z$ a
continuous quadratic mapping. Then $Q\circ F$ is d.c.\ on $C$, provided
at least one of the following conditions is satisfied:
\begin{enumerate}
\item[(a)] $Y$ admits a renorming with modulus of convexity
of power type 2;
\item[(b)] $Z$ is finite-dimensional and $Y$ has the property (D). 
\end{enumerate}
\end{proposition}

\begin{proof}
The case (a) follows immediately from Theorem~\ref{T:best}(c),  
since each
continuous quadratic mapping is $C^{1,1}$.

Suppose (b) holds. We can suppose that $Z=\R^d$ for some $d\in\N$. Then
the components $Q_j$ ($j=1,\ldots,d$) of the quadratic mapping $Q$ are
continuous quadratic forms. Since $Y$ has (D), we can write 
$Q_j=p_j-q_j$ where $p_j,q_j$ are nonnegative continuous quadratic
forms, in particular, they are convex continuous functions that are
bounded on bounded sets. 
By Remark \ref{R:dc}(a) and (b), 
$Q$ is d.c.\ with a
control function which is bounded on bounded subsets of $Y$. Apply
Theorem~\ref{T:best}(a).
\end{proof}

The following Corollary~\ref{bilin} improves \cite[Corollary 4.3.]{VeZa} which
states only that $B\circ(F,G)$ is locally d.c.\ whenever $Y$ and $V$ are Hilbert spaces.

\begin{corollary}\label{bilin}
Let $X,Y,V,Z$ be normed linear spaces, $C\subset X$ an open convex set,
$F\colon C\to Y$ and $G\colon C\to V$ d.c.\ mappings, and
$B\colon Y\times V\to Z$  a continuous bilinear mapping. Then the
mapping $B\circ(F,G)\colon x\mapsto B\bigl(F(x), G(x)\bigr)$ is d.c.\ on
$C$, provided at least one of the following conditions is satisfied:
\begin{enumerate}
\item[(a)] both $Y$ and $V$ admit renormings with modulus of convexity
of power type 2;
\item[(b)] $Z$ is finite-dimensional and $Y\times V$ has the property
(D).
\end{enumerate}
\end{corollary}

\begin{proof}
Observe that $B$ is also a quadratic mapping on $Y\times V$; indeed, it is generated by
the bilinear mapping 
$\tilde{B}\bigl((y,v),(y',v')\bigr)=B(y,v')$ on $(Y\times
V)\times(Y\times V)$. Moreover, by \cite[Lemma~1.7]{VeZa},
the mapping $x\mapsto\bigl(F(x),G(x)\bigr)$ is d.c.\ on $C$.
Apply Fact~\ref{F:powertype}(a) and Proposition~\ref{quadr}.
\end{proof}

\begin{remark}
\begin{enumerate}
\item[(a)] By Fact~\ref{F:powertype}(b), the assumptions in
Proposition~\ref{quadr}(a) and Corollary~\ref{bilin}(a) are satisfied, for instance, if each of
$Y,V$ is isomorphic to a subspace of some $L_p(\mu)$ with $1< p\le 2$
(not necessarily with the same $p$ and $\mu$).
\item[(b)] By \cite[Theorem~1.6 and Observation~3.9]{KoVe}, the
assumptions in
Proposition~\ref{quadr}(b) and Corollary~\ref{bilin}(b) are satisfied, for instance, if each of
$Y,V$ is isomorphic to one (not necessarily the same) of the spaces
$C(K)$, $c_0(\Gamma)$, $L_p(\mu)$ with $2\le p\le\infty$.
\end{enumerate}
\end{remark}

%%%%%%%%%%%%%%%%%%%%%%%%%%%%%%%%%%%%%%%%%%%%%%%%%%%%%%%%%%%%%%%%%%%%%%%%%%%%%%
\section{Global delta-convexity of composed functions}

Here we present positive results which are formulated without using the notion of d.c.\ operators, i.e., those
 which directly concern Hartman's results.
 Probably most interesting is the following immediate consequence of 
Theorem~\ref{T:best}(a).
 
 \begin{theorem}\label{cely}
 Let $X$ be a normed linear space. Let $A\subset X$ be an open  convex set, and
$F\colon A\to \R^n$  and $g: \R^n \to \R$ be d.c. Then the composed function $g\circ F$ is d.c.
\end{theorem}

Since each $C^2$ function $g: \R^n \to \R$ is d.c\ by Proposition \ref{C11} and (I) from Introduction,
 applying Theorem \ref{cely}  to $F=(f,h)$ and $g(x,y)=xy$, we obtain 
that $f\!\cdot\! h$ is d.c.\  
on $A$, whenever $f$ and $h$ are real d.c.\ functions on $A$. 
However, this fact is well-known (cf.~\cite{Hi}) and
  can be proved by a quite elementary way. 
 But the fact that, for instance, $\exp(f)$ and $\frac{fh}{1+f^2 + h^2}$ are  d.c.\ 
 on  $A$  seems to be new. (Hartman's results only imply that these functions are locally d.c.)
  
 For compositions of special d.c.\ functions, we obtain the following.
 
 \begin{theorem}\label{spec}
 Let $X$ be  a normed  linear space and  $A\subset X$, $B\subset \R^n$ convex sets. Let
  $F = (F_1,\dots,F_n)\colon A \to B$ be a d.c.\ mapping and $g\colon B \to \R$ a d.c.\ function.
  Then $g\circ F$ is d.c.\ on $A$, provided
at least one of the following
conditions is satisfied:
\begin{enumerate}
\item[(a)] $A$ is open, $F$ is d.c., and $g$ is a difference of two Lipschitz convex functions;
\item[(b)] each $F_i$ is a difference of two continuous convex functions, which are 
bounded on bounded subsets of $A$, and
the restriction of $g$ to each bounded convex subset of $B$
is a difference of two Lipschitz convex functions;
\item[(c)] $X = \R^k$, $A$ is open or closed, $F$ is d.c., and, for each $a \in A$, there exists $\ep>0$ such  $g$ is a difference of two Lipschitz convex functions on $B \cap B(F(a),\ep)$.
\end{enumerate}
\end{theorem}

\begin{proof}
To prove (a), observe that, by  Fact \ref{F:1} (c), we can suppose that $g$ is a difference of two Lipschitz convex functions
 on the whole $\R^n$. Hence $g\circ F$ is d.c.\ on $A$ by Theorem \ref{cely}. 

The part (b) follows from Remark~\ref{R:dc}(a) and (b), and 
Proposition \ref{nove}.
  
Let (c) hold. By Corollary \ref{C:loc} (or (I)), it is sufficient to show that 
   $g\circ F$ is locally d.c.\ on $A$. To this end, choose an $a \in A$ and find  $\ep>0$ such  $g$ is a difference of two Lipschitz convex functions on $B \cap B(F(a),\ep)$. Since $F$ is continuous, we can find $\delta>0$ such that
    $F(B(a,\delta) \cap A) \subset B \cap B(F(a),\ep)$. Using Proposition \ref{P:veza} (and Remark~\ref{R:dc}(b)), we obtain
     that $g\circ F$ is d.c.\ on  $B(a,\delta) \cap A$.
\end{proof}
 
Note that the case (c) follows also from proofs in \cite{H}. 
However, a claim of
P.~Hartman (see \cite{H}, p.708, lines 12-17), which would imply (via (I)
from Introduction)
that,  in (c), it is sufficient to write ``$g$ is d.c.\ and Lipschitz''
instead of ``$g$ is a difference of two Lipschitz convex functions'', is
false (presumably due to a misprint). 
This is shown by the following example.
%  (the point is that a Lipschitz  d.c. function need not be a difference of two Lipschitz convex functions).
  
  \begin{example}\label{chyba}
  Let $d\colon \R \to \R$ be the characteristic function of the set $S: = \bigcup_{n \in \N} [-2^{-2n+2},-2^{-2n+1})$
   and put $g(x): = \int_{-1}^x d$ for $x\in [-1,0]$. First we will show that $g$ is a Lipschitz d.c.\ function which
   is not a difference of two Lipschitz convex functions on $[-1/2,0]$. Since $d$ is bounded, $g$ is clearly Lipschitz.
    Clearly $g'_+(x) = d(x)$, $x \in [-1,0)$, since $d$ is right continuous. 
For $x\in[-1,0)$,
let $v(x)$ be the total
     variation of $d$ on the interval $[-1,x]$. It is easy to check that $v(x)= n-1$ for $x \in [-2^{1-n}, -2^{-n})$,
      and consequently $\int_{-1}^0 v = \sum_{n=1}^{\infty} (n-1) 2^{-n} < \infty$. Thus both $v$ and
       $w:= v -d$ are nondecreasing and (Lebesgue) integrable on $[-1,0)$. So, $c_1(x):= \int_{-1}^x v$ and
        $c_2(x):= \int_{-1}^x w$ are continuous convex functions on $[-1,0]$, and
        $$ g(x)  = \int_{-1}^x d =  \int_{-1}^x (v-w) = c_1(x) - c_2(x),\ \ \ \ x \in [-1,0].$$
        Therefore, $g$ is d.c.\ on $[-1,0]$. 
        
        Now, suppose to the contrary that $g=p-q$ on $[-1/2,0]$, where $p$, $q$ are convex Lipschitz functions
         on $[-1/2,0]$. It is well-known that then the right derivatives $p'_+$, $q'_+$ are finite, bounded 
and nondecreasing
          functions on $[-1/2,0]$. Further  $d = g'_+ = p'_+ - q'_+$ on $[-1/2,0)$. 
Let $V_a^b\phi$ denote the total variation of $\phi$ on $[a,b]$.
Then, for $ x \in [-1/2,0)$,
            \begin{multline*}
             v(x)- v(-1/2)=V_{-1/2}^x ( p'_+ - q'_+) \leq 
             V_{-1/2}^x \, p'_+  +   V_{-1/2}^x \, q'_+\\
              = \bigl(p'_+(x)- p'_+(-1/2)\bigr) + \bigl(q'_+(x)- q'_+(-1/2)\bigr)=: z(x),
             \end{multline*}
             which is a contradiction, since $\lim_{x\to 0+} v(x) = \infty$ and $z$ is a bounded function.
  
  Now, set $F(x):= -|x|$ for $x \in [-1,1]$. Then $g\circ F$ is not d.c.\ even on $(-1,1)$. Indeed, otherwise $g\circ F$ would
   be  a difference of two Lipschitz convex functions on $[-1/2,0]$, which is not true, since $g\circ F = g$ on $[-1/2,0]$.
\end{example}

%%%%%%%%%%%%%%%%%%%%%%%%%%%%%%%%%%%%%%%%%%%%%%%%%%%%%%%%%%%%%%%%%%%%%%%%%%%%%%%%%%%%%%%%%%%%%%%%

\section{The main counterexample}

%The main result of this section uses some
%ideas from \cite{KM} (where a
%locally d.c.\ function which is not d.c.\ was constructed). 

The main result of this section (Theorem~\ref{ex}) provides
a general construction of non-d.c.\ composed
mappings. Its proof uses some ideas from \cite{KM}.

The following lemma, implicitly contained in
\cite{KM}, is useful
for showing that certain functions or mappings are not d.c.

\begin{lemma}\label{ndc}
Let $X,Y$ be normed linear spaces, let $A\subset X$ be an open convex
set with $0\in A$, and let $F\colon A\to Y$ be a mapping. Suppose there exist
$\lambda\in(0,1)$ and a sequence of balls $B(x_n,\delta_n)\subset A$ such
that $\{x_n\}\subset\lambda A$, $\delta_n\to0$ and $F$ is unbounded on
each $B(x_n,\delta_n)$. Then $F$ is not d.c.\ on $A$.
\end{lemma}

\begin{proof}
Suppose the contrary. Let $f$ be a  control function for $F$
on $A$. We can suppose $f\geq 0$ (otherwise choose an affine function $g$ such that $g \leq f$ on $A$, and consider
 $f-g$ instead of $f$). For each $n$, let $z_n\in A$ be such that $x_n=\lambda z_n$.
Observe that $\|h\|<\delta_n$ implies
$x_n+h=\lambda z_n+(1-\lambda)\frac{h}{1-\lambda}$ and 
$\|\frac{h}{1-\lambda}\|<\frac{\delta_n}{1-\lambda}$. Now, fix $m\in\N$ so
large that $B(0,\frac{\delta_m}{1-\lambda})\subset A$ and both $F$ and $f$ are
bounded on $B(0,\frac{\delta_m}{1-\lambda})$. Then we have
\begin{align}
{\textstyle
\|\lambda F(z_m)+(1-\lambda)F(\frac{h}{1-\lambda})}&{\textstyle -F(x_m+h)\|}
\label{uno}\\
&{\textstyle\le \lambda f(z_m)+(1-\lambda)f(\frac{h}{1-\lambda})-f(x_m+h)}
\notag\\
&\le {\textstyle\lambda f(z_m)+(1-\lambda)f(\frac{h}{1-\lambda})}
\label{due}
\end{align}
whenever $\|h\|<\delta_m$. But this is a contradiction since the
expression \eqref{due} is bounded on $\{h: \|h\| < \delta_m\}$ while 
\eqref{uno} is not (because $F$ in 
unbounded on $B(x_m,\delta_m)$).
\end{proof}

\begin{lemma}\label{strexp}
Let $X$ be a normed linear space. Let $e\in S_X$, $e^*\in S_{X^*}$ and $c>0$
be such that $e^*(e)=1$ and the implication
\begin{equation}\label{hypo}
e^*(u)>1-\epsilon\text{ and } \|u\|\le1\ \ \Rightarrow\ \ 
\|u-e\|\le c\,\epsilon
\end{equation}
holds for $u\in X$ and $\epsilon>0$.
Then the following implication holds for $x\in X$ and $0<\delta<\frac{1}{2}\,$:
\begin{equation}\label{thesis}
{\textstyle\frac{1}{2}}\|x\|^2 < 
{\textstyle\frac{1}{2}}\|e\|^2 + e^*(x-e) +\delta\ \ \Rightarrow\ \ 
\|x-e\|<(1+2c)\sqrt{2\delta}\,.
\end{equation}
\end{lemma}

\begin{proof}
Let $x\in X$ and $0<\delta<\frac{1}{2}$ satisfy the left-hand side of
\eqref{thesis}. 
Then 
\[
\textstyle\frac{1}{2}\|x\|^2<e^*(x)-\frac{1}{2}+\delta\le\|x\|-
\frac{1}{2}+\delta
\]
which implies $\frac{1}{2}(1-\|x\|)^2<\delta$. Thus
$0<1-\sqrt{2\delta}<\|x\|<1+\sqrt{2\delta}.$

If $\|x\|\le1$, then $e^*(x)>\frac{1}{2}\|x\|^2+\frac{1}{2}-\delta>
\frac{1}{2}(1-\sqrt{2\delta})^2+\frac{1}{2}-\delta=1-\sqrt{2\delta}$. By
the assumption \eqref{hypo},
$\|x-e\|\le c\,\sqrt{2\delta}<(1+2c)\sqrt{2\delta}$.

If $\|x\|>1$, then (as above) 
$e^*(\frac{x}{\|x\|})>\frac{1}{\|x\|}\left(\frac{1}{2}\|x\|^2+\frac{1}{2}-\delta\right)
>\frac{1-\sqrt{2\delta}}{\|x\|}>\frac{1-\sqrt{2\delta}}{1+\sqrt{2\delta}}=
1-\frac{2\sqrt{2\delta}}{1+\sqrt{2\delta}}.$ By \eqref{hypo}, we have
$\|\frac{x}{\|x\|}-e\|\le c\,\frac{2\sqrt{2\delta}}{1+\sqrt{2\delta}}$.
Consequently,
$\|x-e\| \le \|x-\frac{x}{\|x\|}\| + \|\frac{x}{\|x\|}-e\| \le 
(\|x\|-1) + \frac{2c\sqrt{2\delta}}{1+\sqrt{2\delta}}<
\sqrt{2\delta}\bigl(1+\frac{2c}{1+\sqrt{2\delta}}\bigr)<(1+2c)\sqrt{2\delta}$.
\end{proof}

\begin{lemma}\label{baze}
For each infinite dimensional normed linear space, there exists a
countable biorthogonal system $\{e_n,e^*_n\}\subset X\times X^*$ such that:
\[
\textstyle
\|e_n\|=1\ (n\in\N),\ \ 
R:=\sup_{n}\|e^*_n\|<\infty,\ \ 
r:=\inf_{m\ne n}\|e_m-e_n\|>0.
\]
\end{lemma}

\begin{proof}
The completion of $X$ contains a normalized basic sequence
$\{e_n\}$ (see \cite[Theorem~6.14]{FHHMPZ}). By the ``small perturbation lemma'' 
\cite[Theorem~6.18]{FHHMPZ}, we may assume that $\{e_n\}\subset X$. Let $e^*_n$
($n\in\N$)
be Hahn-Banach extensions of the corresponding coefficient functionals; it is 
well-known that they are equi-bounded (cf.\ \cite[p.164]{FHHMPZ}). Moreover, for
$m\ne n$, we have $\|e_n-e_m\|\ge 1/R$, since $1 = e^*_n(e_n-e_m) \leq R\, \|e_n-e_m\|$.
% we have
%$\|e_n-e_m\|\ge\mathrm{dist}\bigl(e_n,\mathrm{ker}(e^*_n)\bigr)=
%\|e^*_n\|^{-1}\ge 1/R$.
\end{proof}

\begin{lemma}\label{construction}
Let $X,Y$ be normed linear spaces, $X$ infinite dimensional. 
Then, for
each bounded sequence $\{y_n\}\subset Y$, there exists a d.c.\ mapping
$\Phi\colon X\to Y$ such that:
\begin{enumerate}
\item[(a)] $\Phi=0$ outside $B_X$;
\item[(b)] $\Phi$ admits a control function that is Lipschitz on bounded
sets;
\item[(c)] $\{y_n\}\subset \Phi(B_X)$ and 
$\Phi(X)\subset\mathrm{conv}\bigl[\{0\}\cup\{y_n\}_{n\in\N}\bigr]$.
\end{enumerate}
\end{lemma}

\begin{proof}
Let $\{e_n\}$, $\{e^*_n\}$, $R$ and $r$ be as in Lemma~\ref{baze}.
Observe that $R\ge 1$ since $e_1^*(e_1)=1$.
Fix an arbitrary $\rho\in(0,\frac{1}{R})$. 
The symmetric closed convex set
\[
C:=\overline{\mathrm{conv}}\bigl(
\rho B_X \cup \{\pm e_n\}_{n\in\N}
\bigr)
\] 
is the unit ball of an equivalent norm $|\!|\!|\cdot|\!|\!|$ on $X$
since $\rho B_X\subset C\subset B_X$. 

Fix an arbitrary $n\in\N$. It is easy to see that
$|\!|\!|e^*_n|\!|\!|=\max e^*_n(C)=e^*_n(e_n)=1$, which implies that also
$|\!|\!|e_n|\!|\!|=1$. Let $\epsilon>0$ and $u\in C$ be such that 
\[
e^*_n(u)>1-\epsilon.
\]
Observe that $C=\mathrm{conv}\left(\{e_n\}\cup C_n\right)$ where
\[
C_n=\overline{\mathrm{conv}}\bigl(
\rho B_X \cup \{- e_k\}_{k\in\N} \cup \{e_k\}_{k\in\N\setminus\{n\}}
\bigr).
\]
Thus we can write $u=(1-\lambda)e_n+\lambda v$ where $v\in C_n$ and
$0\le\lambda\le1$. Since 
$$1-\epsilon<e^*_n(u)\le 1-\lambda +\lambda\sup e^*_n(C_n)\le 1-\lambda
+\lambda R\rho,$$
we easily get $\lambda<\frac{\epsilon}{1-R\rho}$. Consequently,
$$
|\!|\!|u-e_n|\!|\!|=
\lambda|\!|\!|v-e_n|\!|\!|\le\frac{2\epsilon}{1-R\rho}\,.
$$

Denote $g(x)=\frac{1}{2}|\!|\!|x|\!|\!|^2$. By Lemma~\ref{strexp}, for
 $n\in\N$, $x\in X$ and $0<\delta<\frac{1}{2}$ the following implication holds:
\[
g(x)< g(e_n)+e^*_n(x-e_n)+\delta\ \ \Rightarrow\ \ 
\|x-e_n\|\le|\!|\!|x-e_n|\!|\!|<(1+{\textstyle \frac{4}{1-R\rho}})\sqrt{2\delta}.
\]
Since the sequence $\{e_n\}$ is uniformly discrete, it is
possible to fix a
$\delta\in(0,\frac{1}{2})$ so small that the open convex sets
\[
D_n=\{x\in X: g(x)< g(e_n)+e^*_n(x-e_n)+\delta\}
\]
satisfy $\mathrm{dist}_{\|\cdot\|}(D_m,D_n)>\delta$ whenever $m\ne n$.
We have $e_n\in D_n$ for each $n$.

Define $H\colon X\to Y$ by
\[
H(x)=\begin{cases}
{\textstyle\frac{1}{\delta}}\bigl[
g(e_n)+e^*_n(x-e_n)+\delta-g(x)
\bigr]y_n &\text{if $x\in D_n$;}\\
0 &\text{for $x\notin\bigcup_{n\in\N}D_n$.}
\end{cases}
\]
It is easy to see that $H$ is continuous since we have
\begin{equation}\label{nadn}
\textstyle
H(x)=\frac{1}{\delta}\bigl[
\max\{g(x), g(e_n)+e^*_n(x-e_n)+\delta\} -g(x)
\bigr]y_n\,,\ \ 
x\in D_n+\delta B_X.
\end{equation}
Put $s:=\sup_{n\in\N}\|y_n\|$.
We claim that the formula
\begin{equation}\label{kontr}
\textstyle
h(x)=\frac{s}{\delta}\,\sup_{n\in\N}
\bigl(\max\{g(x),g(e_n)+e^*_n(x-e_n)+\delta\}\bigr)\,
+ \frac{s}{\delta}\,g(x)
\end{equation}
defines a control function for $H$, which is Lipschitz on bounded sets.
First, observe that $h(0)=\frac{s}{\delta}\max\{0,\frac{1}{2}-
1+\delta\}=0$. Moreover, since $g$ is Lipschitz on bounded sets and the
functionals $e^*_n$ ($n\in\N$) are equi-Lipschitz, \eqref{kontr} defines
a real convex function that is Lipschitz on bounded sets.
Fix $y^*\in B_{Y^*}$. To prove that the function $\psi:=y^*\circ H +h$
is convex, it is sufficient to show that it is locally convex. 
For
$x\notin\bigcup_n\overline{D_n}=\overline{\bigcup_n D_n}$, we have
$\psi(x)=h(x)$.
For $x\in D_n+\delta B_X$, we have $g(x)\ge g(e_k)+e^*_k(x-e_k)+\delta$
whenever $k\ne n$, and hence
\[
\textstyle
h(x)=\frac{s}{\delta}\,
\max\{g(x),g(e_n)+e^*_n(x-e_n)+\delta\}\,
+ \frac{s}{\delta}\,g(x)\,,\ \ x\in D_n+\delta B_X.
\]
Consequently, \eqref{nadn} implies that, on the set $D_n+\delta B_X$, the function
\[
\textstyle
\psi(x)=
\frac{s+y^*(y_n)}{\delta}\,\max\{g(x),g(e_n)+e^*_n(x-e_n)+\delta\}\,
+ \frac{s-y^*(y_n)}{\delta}\,g(x)
\]
is convex (since it is a sum of convex functions).

Observe that $H(e_n)=y_n$. Moreover, for each $x\in D_n$,
\begin{multline*}
\textstyle
0<g(e_n)+e^*_n(x-e_n)+\delta-g(x)\le
\frac{1}{2}+|\!|\!|x|\!|\!|-1+\delta-\frac{1}{2}|\!|\!|x|\!|\!|^2
\\
\textstyle  =
\delta-\frac{1}{2}\bigl(|\!|\!|x|\!|\!|-1\bigr)^2\le\delta\,.\ \ \ 
\end{multline*}
Thus, for each $n$, the image $H(D_n)$ is contained in the
segment $[0,y_n]$. Since the support of $H$ is contained in $2B_X$, the
mapping $\Phi(x):=H(2x)$ has all the required properties (note that $\vf(x):= h(2x)$ clearly controls
$\Phi$, cf. \cite[Lemma 1.5]{VeZa}). 
\end{proof}

\begin{theorem}\label{ex}
Let $X,Y,Z$ be normed linear spaces, $X$ infinite dimensional. Let
$A\subset X$ be an open convex set, let $B\subset Y$ be a convex set, and let 
$G\colon B\to Z$ be a mapping which is unbounded on a bounded subset of
$B$. Then there exists a d.c.\ mapping $F\colon A\to B$ such that
$G\circ F$ is not d.c.\ on $A$.
\end{theorem}

\begin{proof}
We can (and do) suppose that $0\in A$.
Fix $r\in(0,1)$ such that $B(0,2r)\subset A$. By  \cite{BFV}, there exists a continuous convex function $h$ on $X$ such that
 $h(0) =0$ and $\sup_{x \in B(0,r)} \, h(x) = \infty$. For $k \in \N$, set 
$$A_k := \{x \in A:\ h(x) < k,\; \|x\|<k\}.$$ 
Clearly
  each $A_k$ contains $0$, is open and convex; moreover,  $A_k \nearrow A$. 
It is easy to see that, for each $k \in \N$, we can choose
   $v_k \in B(0,r)$ and $0 < \delta_k < 1/k$ such that $B(v_k, 2\delta_k) \subset A_{k+1}\setminus A_k$.

We can (and do) suppose that $0\in B$.
Let $\{y_n\}\subset B$ be a bounded sequence such that
$\|G(y_n)\|\to\infty$, and let $\Phi$ be the corresponding mapping from
Lemma~\ref{construction}.
For each $k\in\N$, define $F_k\colon X\to Y$ by 
\[
F_k(x)=\Phi\left(\frac{x-v_k}{\delta_k}\right).
\]
Since the supports of these mappings are pairwise disjoint and each
$A_k$ intersects only finitely many of them, the mapping
\[
F\colon A\to Y\,,\qquad
F(x):=\sum_{k\in\N} F_k(x)
\]
is well-defined and continuous. Observing that $\vf_n(x):= \vf(\frac{x-v_k}{\delta_k})$ controls $F_k$
 if $\vf$ controls $\Phi$ (cf.  \cite[Lemma 1.5]{VeZa}), we obtain that $F$ is d.c.\ on each $A_k$ with a
Lipschitz (hence bounded) control function. By Proposition~\ref{P:HKM}, $F$ is d.c.\ on $A$.
Moreover,
$F(A)\subset \bigcup_k F_k(X)\subset B$ by Lemma~\ref{construction}(c).
Since $G\circ F$ is unbounded on each $B(v_k,\delta_k)$ and
$v_k\in\frac{1}{2}A$, Lemma~\ref{ndc} implies that
$G\circ F$ is not d.c.\ on $A$.
\end{proof}

\begin{corollary}\label{jldc}
Let $X$ be an infinite dimensional normed linear space, and $A\subset X$ a nonempty open convex set.
\begin{enumerate}
\item[(a)] There exists a positive d.c.\ function $f$ on $A$ such that $1/f$ is not d.c.
\item[(b)] There exists a locally d.c.\ function $g$ on $A$, which is not d.c.
\end{enumerate}
\end{corollary}

\begin{proof}
Applying Theorem~\ref{ex} with $B=(0,\infty)$ and $G(y)=1/y$, we obtain (a). Now, (b)
follows from (a), since $g :=1/f$ is locally d.c.\ by Proposition~\ref{lokdc} 
(or (II) in Introduction).
\end{proof}

%%%%%%%%%%%%%%%%%%%%%%%%%%%%%%%%%%%%%%%%%%%%%%%%%%%%%%%%%%%%%%%%%%%%%%%%%%%%%%%%%%%%%%%%

\bigskip
\bigskip

\subsection*{Acknowledgment}
The research of the first author was partially supported by the Ministero
dell'Universit\`a e della Ricerca of Italy.
The research of the second author was partially supported by the grant
GA\v CR 201/06/0198  from the Grant Agency of
Czech Republic and partially supported by the grant MSM 0021620839 from
 the Czech Ministry of Education.

\end{document}